\documentclass{amsart}
    \usepackage{amsmath}
    \usepackage{amssymb,amsfonts}
    \usepackage[all,arc]{xy}
    \usepackage{enumerate}
    \usepackage{mathrsfs}
    \usepackage{amsthm}
    \usepackage{enumitem}
    \usepackage{extarrows}
    \usepackage{verbatim}
    \usepackage{dsfont}
    \usepackage{tikz-cd}
    \usepackage{tensor}
    \usepackage{mathtools}

    \newtheorem{thm}{Theorem}[section]
    \newtheorem{cor}[thm]{Corollary}
    \newtheorem{prop}[thm]{Proposition}
    \newtheorem{lem}[thm]{Lemma}
    \newtheorem{conj}[thm]{Conjecture}

    \theoremstyle{definition}
    \newtheorem{defn}[thm]{Definition}
    \newtheorem{defns}[thm]{Definitions}

    \theoremstyle{remark}
    
    \newtheorem{rem}[thm]{Remark}

    \newcommand{\Z}{\mathbb{Z}}
    
    \newcommand{\C}{\mathbb{C}}
    \newcommand{\pP}{\mathbb{P}}

    \newcommand{\cO}{\mathcal{O}}

    \newcommand{\cH}{\mathcal{H}}

    \newcommand{\ke}{\mathcal{K}}
    \newcommand{\im}{\mathcal{I}}

    \newcommand{\Gr}{\mathrm{Gr}}

    \newcommand{\Proj}{\mathrm{Proj}}
    \newcommand{\Spec}{\mathrm{Spec}}

    \newcommand{\Hom}{\mathrm{Hom}}
    \newcommand{\Ext}{\mathrm{Ext}}
    \newcommand{\rank}{\mathrm{rank}}

    \newcommand\smvee{\raise0.9ex\hbox{$\scriptscriptstyle\vee$}}
    \makeatletter

    \newcommand*{\rom}[1]{\expandafter\@slowromancap\romannumeral #1@}
    
    \makeatletter
    
    \let\c@equation\c@thm
    \makeatother
    \numberwithin{equation}{section}

\begin{document}

\title{Compactification of the moduli space of minimal instantons on the Fano 3-fold $V_5$}

\author{Xuqiang QIN}
\address{Department of Mathematics, Indiana University, 831 E. Third St., Bloomington, IN 47405, USA}
\email{qinx@iu.edu}


\begin{abstract}
    We study semistable sheaves of rank $2$ with Chern classes $c_1=0$, $c_2=2$ and $c_3=0$ on the Fano 3-fold $V_5$ of Picard number $1$, degree $5$ and index $2$. We show that the moduli space of such sheaves has a component that is  isomorphic to $\mathbb{P}^5$ by identifying it with the moduli space of semistable quiver representations. This provides a natural smooth compactification of the moduli space of minimal instantons, as well as Ulrich bundles of rank $2$ on $V_5$.
\end{abstract}
    
\maketitle
\section{Introduction}

Instanton bundles first appeared in \cite{AHDM}  as a way to describe Yang-Mills instantons on a 4-sphere $S^4$. They provide extremely useful links between mathematical physics and algebraic geometry. The notion of mathematical instanton bundles was first introduced on $\pP^3$. Since then the irreducibility\cite{T}, rationality\cite{MT} and smoothness\cite{JV} of their moduli spaces were heavily investigated. Faenzi\cite{Fa} and Kuznetsov\cite{Ku2} generalized this notion to Fano threefolds, we recall
\begin{defn}\cite{Ku3}
    Let $Y$ be a Fano threefold of index $2$. An \emph{instanton bundle of charge $n$} on $Y$ is a stable vector bundle $E$ of rank $2$ with $c_1(E)=0, c_2(E)=n$, enjoying the instantonic condition: 
    \begin{align*}
        H^1(Y,E(-1))=0.
    \end{align*}
\end{defn}
We mention that the charge $c_2(E)\geq 2$ \cite[Corollary 3.2]{Ku2}. The instanton bundles of charge $2$ are called the \emph{minimal instantons}. 

In this paper, we will be interested in minimal instantons and compactifications of their moduli space on the Fano threefold $V_5$ of degree $5$ and index $2$. Such a threefold is obtained by taking a general codimension $3$ linear section in the Pl\"ucker embedding of $\mathrm{Gr}(2,5)$. The moduli spaces of instanton bundles on $V_5$ were discussed in \cite{Ku3} using quadric nets. They are in general not projective.

On the other hand, Ulrich bundles are defined as vector bundles on a smooth projective variety $X$ of dimension $d$ so that 
\begin{align*}
        H^*(X,E(-t))=0
    \end{align*}
for all $t=1,\ldots,d$. They first appeared in commutative algebra and entered the world of algebraic geometry via \cite{ES}. The existence and moduli space of Ulrich bundles provide great amount of information about the original variety. For example, in the case when $X$ is a smooth hypersurface, the existence of Ulrich bundle is equivalent to the fact that $X$ can be defined set-theoretically by a linear determinant\cite{Beau}. Inspired by \cite{Ku2}, Lahoz, Macr\`i and Stellari\cite{LMS},\cite{LMS2} studied moduli spaces of Ulrich bundles on cubic threefolds and fourfolds using derived categories.
In their recent paper\cite{LP}, Lee and Park described the moduli space of Ulrich bundles on $V_5$. As an easy consequence of their work, we see on $V_5$, (minimal) instanton bundles and Ulrich bundles of rank $2$ differ only by a twist of the very ample divisor $\cO_{V_5}(1)$. Thus they share the same moduli space and compactifications.

Our first result is the classification of semistable rank $2$ sheaves with Chern classes $c_1=0$, $c_2=2$ and $c_3=0$ on $V_5$. \cite{D} classified such sheaves on cubic threefolds and proved that their moduli space is isomorphic to the blow-up of the intermediate Jacobian in the Fano surface of lines. We mimic his method of classification and prove that a parallel classification happens on $V_5$.
\begin{thm}
     Let $E$ be a semistable rank $2$ sheaf with Chern classes $c_1(E)=0$, $c_2(E)=2$ and $c_3(E)=0$ on $V_5$. If $E$ is stable, then either $E$ is locally free or $E$ is associated to a smooth conic $C\subset V_5$ such that we have an exact sequence:
     \begin{align*}
         0\to E\to H^0(\theta(1))\otimes \cO_{V_5}\to \theta(1)\to 0
     \end{align*}
     where $\theta$ is the theta-characteristic of $C$.\\
     If $E$ is strictly semistable, then $E$ is the extension of two ideal sheaves of lines.
\end{thm}
Unfortunately, the method to study the moduli space in \cite{D} does not transfer well to $V_5$. However, we note that $\mathcal{D}^b(V_5)$ has a semi-orthogonal decomposition:
\begin{align*}
    \mathcal{D}^b(V_5)=\langle \mathcal{U}, \mathcal{Q}^{\vee},\cO_{V_5},\cO_{V_5}(1)\rangle
\end{align*}
where $\mathcal{U}$ is the restriction of the universal subbundle and $\mathcal{Q}$ the universal quotient bundle. It is well-known that the subcategory $\mathcal{B}_{V_5}\coloneqq\langle \mathcal{U},\mathcal{Q}\smvee\rangle$ is equivalent to the derived category of finite dimensional representations of the Kronecker quiver $Q_3$ with three arrows:
\begin{align*}
    \Psi:\mathcal{B}_{V_5}\simeq \mathcal{D}^b(Q_3).
\end{align*}
Our next result establish the relation between the sheaves on $V_5$ and representations of $Q_3$:
\begin{thm}
    Let $R$ be a $(-1,1)$-semistable representation of $Q_3$ with dimension vector $(2,2)$, then the $\Psi^{-1}$-induced complex
    \begin{align*}
    C_R: \mathcal{U}^{\oplus 2}\xrightarrow{f_R}\mathcal{Q}\smvee^{\oplus 2}
\end{align*}
in $\mathcal{D}^b(V_5)$ where we put the second term in degree $0$ is isomorphic to a semistable rank $2$ sheaf $E$ with Chern classes $c_1(E)=0$, $c_2(E)=2$ and $c_3(E)=0$.
\end{thm}
Using this relation, we construct a morphism from the moduli space of semistable representations $M_{rep}$ to the moduli space $M$ of semistable rank $2$ sheaves with Chern classes $c_1=0$, $c_2=2$ and $c_3=0$ and prove:
\begin{thm}
    There exists a morphism $\phi\colon M_{rep}\to M$ which identifies $M_{rep}$ with a connected component of $M$. As a result, the moduli space of semistable rank $2$ sheaves with Chern classes $c_1=0$, $c_2=2$ and $c_3=0$ on $V_5$ has a connected component isomorphic to $\pP^5$.
\end{thm}
We believe our results can be generalized to find compactifications of moduli spaces of Ulrich bundles of higher ranks on $V_5$. Also similar ideas should work in finding the moduli space of instanton sheaves on Fano threefolds other than $V_5$ and cubics. In fact in \cite{Q2}, the author proved that moduli space of semistable rank $2$ sheaves with Chern classes $c_1=0$, $c_2=2$ and $c_3=0$ on the degree $4$ Fano threefold $V_4$ is isomorphic to the moduli space of semistable rank 2 even degree vector bundles on a genus 2 curve.

This paper is organized as follows. In the second section the reader can find some preliminary definitions and results that are used throughout the paper. In the third section we classify semistable rank $2$ sheaves with Chern classes $c_1=0$, $c_2=2$ and $c_3=0$ on $V_5$, showing the parallel result as on cubic threefolds holds. In the fourth section we connect such sheaves to representations of the Kronecker quiver using derived category. In the last section we describe the moduli spaces.
\subsection*{Notations and conventions}
\begin{itemize}
    \item We work over the complex numbers $\C$.
    \item Let $E$ be a sheaf on $V_5$, we use $E_{tor}$ to denote the torsion part of $E$ and $E_{tf}$ to denote the torsion-free quotient $E/E_{tor}$.
    \item Let $E$ be a sheaf on $V_5$. We use $H^i(E)$ to denote $H^i(V_5,E)$ for simplicity. Also we use $h^i(E)$ to denote the dimension of $H^i(V_5,E)$ as a complex vector space.
    \item Let $F$ be a sheaf or a representation with certain characterization, we will use $[F]$ to denote the point it corresponds to in the moduli space.
\end{itemize}
\subsection*{Acknowledgement}
The author would like to thank his advisor Valery Lunts for constant support and inspiring discussions. He would like to thank Michael Larsen, Wai-Kit Yeung, Hui Yu, Shizhuo Zhang for useful discussions. He would like to thank Kyeong-Dong Park and the anonymous referee for many helpful comments.
\section{Preliminaries}
    
\subsection{Fano 3-fold $V_5$}
Let $V$ be a complex vector space of dimension $5$, let $A\subset \Lambda^2V^*$ be a $3$-dimensional subspace of $2$-forms on $V$. It is well known that if $A$ is generic, then the intersection of $\Gr(2,V)$ with $\pP(A^{\perp})$ in $\pP(\Lambda^2V)$ is a smooth Fano threefold of Picard number $1$, degree $5$ and index $2$. Varying $A$ generically will provide projectively equivalent varieties. We use $V_5$ to denote this unique smooth threefold. Note $V_5$ comes with a natural choice of a very ample line bundle $\cO_{V_5}(1)$.  We will always use this polarization for the rest of this paper. Let $S\in |\cO_{V_5}(1)|$ be a generic hyperplane section, then $S$ is a smooth del Pezzo surface of degree $5$.\\

The cohomology group of $V_5$ is isomorphic to $\Z^4$:
\begin{align*}
    H^*(V_5,\Z)&=H^0(V_5,\Z)\oplus H^2(V_5,\Z)\oplus H^4(V_5,\Z)\oplus H^6(V_5,\Z)\\
    &=\Z[V_5]\oplus\Z[h]\oplus\Z[l]\oplus\Z[p]
\end{align*}
with $h\cdot l=p,h^2=5l,h^3=5p$.\\
Let $\mathcal{U}$ be the restriction of the universal subbundle on $\Gr(2,5)$ to $V_5$ and $\mathcal{Q}$ be the restriction of the universal quotient bundle, we have an exact sequence:
\begin{align*}
    0\to \mathcal{U}\to \C^5\otimes\cO_{V_5}\to \mathcal{Q}\to 0.
\end{align*}
$\mathcal{U}$ and $\mathcal{Q}\smvee$ will play important roles in this paper. We first note $\Hom(\mathcal{U},\mathcal{Q}\smvee)=A$ (see \cite{O}). Their cohomology groups were computed in \cite[Lemma 2.9]{LP}:
\begin{lem}
    The cohomology groups $H^i(V_5,\mathcal{U}(j))$ and $H^i(V_5,\mathcal{Q}\smvee(j))$ for $j=-2,-1,0$ are as follows:
    \begin{align*}
        &H^*(V_5,\mathcal{U})=H^*(V_5,\mathcal{U}(-1))=0\\
        &H^*(V_5,\mathcal{Q}\smvee)=H^*(V_5,\mathcal{Q}\smvee(-1))=0\\
        &H^i(V_5,\mathcal{U}(-2))=
        \begin{cases}
            \C^5 &\text{if i=3}\\
            0   &\text{otherwise}
        \end{cases}\\
         &H^i(V_5,\mathcal{Q}\smvee(-2))=
        \begin{cases}
            \C^5 &\text{if i=3}\\
            0   &\text{otherwise}
        \end{cases}
    \end{align*}
\end{lem}
The Chern classes of $\mathcal{U}$ and $\mathcal{Q}\smvee$ are as follows:
\begin{align*}
    &\rank(\mathcal{U})=2, c_1(\mathcal{U})=-h, c_2(\mathcal{U})=2l\\
    &\rank(\mathcal{Q}\smvee)=3, c_1(\mathcal{Q}\smvee)=-h, c_2(\mathcal{Q}\smvee)=3l, c_3(\mathcal{Q}\smvee)=-p
\end{align*}
Moreover, we have
\begin{align*}
    \mathrm{td}(\mathcal{T}_{V_5})&=1+h+\frac{8}{3}l+p.
\end{align*}
The Fano variety of lines on $V_5$ is $\pP(A)\simeq\pP^2$. It parametrizes the  ideal sheaves of lines in $V_5$. Moreover each ideal sheaf corresponds to a representation of $Q_3$ with dimension vector $(1,1)$(see section $2.4$), thanks to the following result:
\begin{lem}\cite[Lemma 4.2]{Ku2}
    For each point $a\in \pP(A)$, we have an exact sequence:
    \begin{align}
        0\to \mathcal{U}\xrightarrow{a} \mathcal{Q}\smvee\to I_{L}\to 0
    \end{align}
    where $L$ is the line corresponding to $a$.
\end{lem}
The restriction of any vector bundles on $V_5$ to a line $L$ splits into direct sum of line bundles, we have:
\begin{lem}\cite[Lemma 2.17]{Sa}
For any line $L$ in $V_5$, we have
\begin{align*}
    \mathcal{U}|_L&=\cO_L(-1)\oplus\cO_L\\
    \mathcal{Q}\smvee|_L&=\cO_L(-1)\oplus\cO_L\oplus\cO_L.
\end{align*}
\end{lem}
Since a smooth conic $C$ is isomorphic to $\pP^1$, any restriction of vector bundle also splits. We clarify our notation here: we will use $\theta$ to denote the theta characteristic of $C$, which is the inverse of the ample generator the Picard group of $C$. By $\cO_C(1)$ we mean the very ample line bundle corresponding to the embedding of the conic. It is twice the ample generator in the Picard group.
\begin{lem}\cite[Lemma 2.40]{Sa}
    Let $C$ be any smooth conic in $V_5$. We have 
    \begin{align*}
        \mathcal{U}|_C&=\theta\oplus\theta,\\
    \mathcal{Q}\smvee|_C&=\theta\oplus\theta\oplus\cO_C.
    \end{align*}
\end{lem}

\subsection{Stability of sheaves} Let $X$ be a smooth projective variety of dimension $n$ and $\cO_X(1)$ be a fixed ample line bundle. Let $E$ be a coherent sheaf of rank $r$, then the slope of $E$ is defined as:
\begin{align*}
    \mu(E)=\frac{c_1(E)\cdot c_1(\cO_X(1))^{n-1}}{rc_1(\cO_X(1))^{n}}
\end{align*}
For a torsion-free sheaf $E$, the \emph{reduced Hilbert polynomial} $p(E)$ is defined by
\begin{align*}
    p(E,n)=\frac{\chi(E(n))}{\rank(E)}
\end{align*}
The sheaf $E$ is called \emph{(semi)stable} if it is torsion-free and for any torsion-free subsheaf $F\subset E$ with $0<\rank(F)<\rank(E)$, we have 
\begin{align*}
    p(F)(\leq)<p(E)
\end{align*}
for $n\gg0$. \\
The sheaf $E$ is called \emph{$\mu$-(semi)stable} if it is torsion-free and for any torsion-free subsheaf $F\subset E$ with $0<\rank(F)<\rank(E)$, we have 
\begin{align*}
    \mu(F)(\leq)<\mu(E)
\end{align*}
We have the following implications:
\begin{align*}
    \mu-\text{stable}\Rightarrow \text{stable} \Rightarrow \text{semistable}\Rightarrow \mu-\text{semistable}
\end{align*}
A very useful criterion for stability of vector bundles is due to Hoppe:
\begin{lem}
    Assume the Picard group of $X$ is $\Z$ and its ample divisor $\cO_X(1)$ has global sections. Let $F$ be a vector bundle of rank $r$ on $X$ so that for each $1\leq k\leq r-1$, $(\Lambda^k(F))_{norm}$ has no global sections. Then $F$ is $\mu$-stable.
\end{lem}
\noindent Here, for a sheaf $F$, $F_{norm}$ is the unique twist $F(n)$ such that $-1<\mu(F(n))\leq 0$. Using this result, we can easily check
\begin{lem}
    The vector bundles $\mathcal{U},\mathcal{U}\smvee , \mathcal{Q},\mathcal{Q}\smvee$ are all $\mu$-stable.
\end{lem}

We also recall the following useful results:

\begin{prop}\cite{H1}\label{hlemma}
  Let $X$ be a smooth projective variety of dimension at least $2$ and $E$ a vector bundle of rank $2$ on $X$. If there exists a global section of $E$ whose zero locus is pure codimension $2$, then we have a short exact sequence:
  \begin{align*}
      0\to \cO_X\to E\to I_Y\otimes \mathrm{det}(E)\to 0.
  \end{align*}
\end{prop}

\begin{prop}[Mumford-Castelnuovo Criterion]\label{mumford}
  Let $F$ be a coherent sheaf on a projective variety $X$. Suppose $h^i(X,F(-i))=0$ for all $i\geq 1$, then $h^i(X,F(k))=0$ for all $i\geq 1$ and $k\geq -i$. Moreover $F$ is generated by global sections.
\end{prop}

\begin{thm}[The Serre construction]\label{Serre}(See \cite{Ar})
    Suppose $X$ is a smooth projective variety over $\C$. Let $L$ be a invertible sheaf so that $h^1(L^{-1})=0$ and $h^2(L^{-2})=0$ and $Y\subset X$ a closed subscheme of pure codimension $2$. We have an isomorphism $\Ext^1(I_Y\otimes L,\cO_X)=H^0(\cO_Y)$. The subscheme $Y$ is the zero locus of a section of a vector bundle of rank $2$ with determinant $L$ if and only if $Y$ is locally complete intersection and $\omega_Y=(\omega_X\otimes L)|_Y.$
\end{thm}

\subsection{Derived Categories} Let $X$ be an algebraic variety over $\C$, we use $\mathcal{D}^b(X)$ to denote the bounded derived categories of coherent sheaves on $X$. We denote $\Ext^p(F,G)=\Hom(F,G[p])$ and $\Ext^\bullet(F,G)=\oplus_{p\in\Z}\Ext^p(F,G)[-p]$. Recall that objects $E_1,\ldots,E_n$ in the bounded derived category of coherent sheaves $\mathcal{D}^b(X)$ forms a full exceptional collection if 
    \begin{enumerate}
          \item $\Hom(E_i,E_i[m])=\C$ if $m=0$ and is $0$ otherwise;
          \item $\Hom(E_i,E_j[m])=0$ for all $m\in\Z$ if $j<i$;
          \item The smallest triangulated subcategory of $\mathcal{D}^b(X)$ containing $E_1,\ldots,E_n$ is itself.
    \end{enumerate}
     An exceptional collection is strong if in addition:
    $\Hom(E_i,E_j[m])=0$ for all $i,j$ if $m\neq 0$. \\
    On a Fano variety $X$, Kodaira vanishing theorem implies:
    \begin{align*}
        H^i(X,\cO_X)=0
    \end{align*}
    for all $i>0$. Thus all line bundles on $V_5$ are exceptional objects. Moreover \cite{O} showed that $\mathcal{D}^b(V_5)$ has a full exceptional collection:
\begin{align*}
    \mathcal{D}^b(V_5)=\langle \mathcal{U}, \mathcal{Q}\smvee,\cO_{V_5},\cO_{V_5}(1)\rangle
\end{align*}
We use $\mathcal{B}_{V_5}$ to denote the triangulated subcategory $\langle \mathcal{U},\mathcal{Q}\smvee\rangle=\langle\cO_{V_5},\cO_{V_5}(1)\rangle^{\perp}$. The following result can be found in \cite{Ku2}.
\begin{lem}\label{repcompute}
    On the Fano threefold $V_5$ of index $2$, degree $5$ and Picard rank $1$, 
    we have canonical isomorphism
    \begin{align*}
         \Ext^\bullet(\mathcal{U},\mathcal{Q}\smvee)=A
    \end{align*}
    Here $A$ is the $3$-dimensional subspace of two forms we chose to define $V_5$ in Section 2.1. As a result, we have an equivalence of category
    \begin{align*}
        \Psi:\mathcal{B}_{V_5}\simeq \mathcal{D}^b(Q_3).
    \end{align*}
    where $\mathcal{D}^b(Q_3)$ is the derived category of finite dimensional representations of the Kronecker quiver with $2$ vertices and $3$ arrows from the first vertex to the second.
\end{lem}
For $F\in \mathcal{B}_{V_5}$, $\Psi(F)$ is the representation $(M_1^\bullet,M_2^\bullet)$ with
\begin{align*}
    M_1^\bullet&=\Ext^\bullet(F,\mathcal{U}[1])^*\\
    M_2^\bullet&=\Ext^\bullet(\mathcal{Q}\smvee,F).
\end{align*}.
\subsection{Quivers representations and their moduli} A quiver $Q$ is given by two sets $Q_{vx}$ and $Q_{ar}$, where the first set is the set of vertices and the second is the set of arrows, along with two functions $s,t:Q_{ar}\to Q_{vx}$ specifying the source and target of an arrow. The path algebra $\C Q$ is the associative $\C$-algebra whose underlying vector space has a basis consisting of elements of $Q_{ar}$. The product of two basis elements is defined by concatenation of paths if possible, otherwise $0$. The product of two general elements is defined by extending the above linearly.\
    
Let $Q$ be a quiver. A quiver representation $R=(R_v,r_a)$ consists of a vector space $R_v$ for each $v\in Q_{vx}$ and a morphism of vector spaces $r_a:R_{s(a)}\to R_{t(a)}$ for each $a\in Q_{ar}$. A subrepresentation of $R$ is a pair $R'=(R_v',r_a')$ where $R_v'$ is a subspace of $R_v$ for each $v\in Q_{vx}$ and $r'_a:R'_{s(a)}\to R'_{t(a)}$ is a morphism of vector spaces for each $a\in Q_{ar}$ such that $$r'_a=r_a|_{R'_{s(a)}}$$ and 
    \begin{equation}\label{subset}
        r_a(R'_{s(a)})\subset R'_{t(a)}.
    \end{equation}
    Thus we have the commutative diagram
    \[ \begin{tikzcd}
      R'_i \arrow[r,"r'_a"] \arrow{d}{\iota_i} & R'_j \arrow{d}{\iota_j} \\%
      R_i \arrow{r}{r_a}& R_j
      \end{tikzcd}
      \]
    for any arrow $a$ from $i$ to $j$.
    We use $R'\subset R$ to denote that $R'$ is a subrepresentation of $R$.

    Given a quiver $Q$, a weight is an element $\Theta\in\Z^N$ where $N=|Q_{vx}|$.
    For a weight $\Theta$, the weight function is defined by:
    \begin{equation*}
        \Theta(S)=\sum_{i=1}^N d_i \Theta_i,
    \end{equation*}
    where $S$ is a representation of $Q$ and $d_i$ and $\Theta_i$ are the $i$-th entries of $\vec{d}$ and $\Theta$ respectively. We recall the definition of semi-stability:
    \begin{defn}
        A representation $R$ is \emph{$\Theta$-semistable} if for any subrepresentation $R'\subset R$
    \begin{equation*}
        \Theta(R')\geq\Theta(R)
    \end{equation*}
    $R$ is \emph{$\Theta$-stable} if all the above inequalities are strict.
    \end{defn}
    
    Let $\vec{d}$ be a dimension vector, the set of representations of $Q$ with dimension vector $\vec{d}$ forms an affine space $\mathrm{Rep(Q)}_{\vec{d}}$. For a weight $\Theta$, the set of $\Theta$-semistable representations forms an open subscheme $\mathrm{Rep(Q)}^{\Theta-ss}_{\vec{d}}$ of $\mathrm{Rep(Q)}_{\vec{d}}$, the set of $\Theta$-stable representations forms an open subscheme $\mathrm{Rep(Q)}^{\Theta-s}_{\vec{d}}$ of $\mathrm{Rep(Q)}^{\Theta-ss}_{\vec{d}}$.
    
    The group $G_0=\Pi_iGL(d_i)$ acts by incidence on $\mathrm{Rep(Q)}$, in other words, it acts by $(g\cdot a)=g_{t(a)}r_ag^{-1}_{s(a)}$. Apparently, the diagonal subgroup $\C_{\text{diag}}^*$ of $(\C^*)^{Q_{vx}}$ consisting of elements of the form $(k,k,\ldots,k)$ for $k\in\C^*$ acts trivially on $\mathrm{Rep(Q)}$. So it is natural to only consider the action of $G=\Pi_iGL(d_i)/\C^*_{\mathrm{diag}}$.
    \begin{defn}
      Two representations with dimension vector $\vec{d}$ are \emph{isomorphic} if they are in the same orbit under the action of $G$.\
    \end{defn}

    Give a weight $\Theta$, we may interpret $\Theta$ as a multiplicative character of the group $G$ by the formula $g\mapsto \mathrm{det}(g_1)^{\Theta_1}\ldots\mathrm{det}(g_n)^{\Theta_n}$. The moduli space of $\Theta$-semistable representations with dimension vector $\vec{d}$ is the GIT quotient
    \begin{align*}
        M^{\Theta-ss}_{\vec{d}}:&=\mathrm{Rep(Q)}_{\vec{d}}//_\Theta G.
    \end{align*}
    We mention a few facts about $M_\Theta$. 
    An equivalent definition of $M^{\vec{d}-ss}$ is to consider the graded ring
       $$ B_{\Theta}= \bigoplus_{r\geq 0}B(r\Theta)$$
    where $B(r\Theta)$ is the space of $r\Theta$-semi-invariant functions in the coordinate ring of $\mathrm{Rep(Q)}$. Then the GIT quotient is defined as 
    \begin{equation*}
        M_\Theta=\Proj(B_\Theta)
    \end{equation*}
    From this definition, it is easy to see that $M_{\Theta}$ is an irreducible normal projective scheme. 

\subsection{Representations of Kronecker quiver $Q_3$}
Let $Q_3$ be the Kronecker quiver with two vertices $1,2$ and three arrows $r_1,r_2,r_3$ from $1$ to $2$. We let $\Theta=(-1,1)$ be a weight for $Q_3$. We will always use this weight when we talk about stability of quiver from now on. We know that the moduli space of $\Theta$-semistable representations of $Q_3$ with dimension vector $(1,1)$ is isomorphic to $\pP^2$. Also, it is well-known that the moduli space of $\Theta$-semistable representation of $Q_3$ with dimension vector $(2,2)$ is isomorphic to $\pP^5$. We provide a proof that will be useful later on.
\begin{prop}\label{coor}
We have
  $$M^{\Theta-ss}_{(2,2)}\cong \pP^5$$
\end{prop}
\begin{proof}
(cf. \cite{ALB} page 2.)A representation of $Q_3$ with dimension vector $(2,2)$ is given by three $2\times 2$ matrices
    \[
    Y_i=
    \begin{bmatrix}
        a_i & b_i\\
        c_i & d_i
    \end{bmatrix}
    \]
     for $i=1,2,3$, each corresponding to the arrow $r_i$. Then the affine representation scheme $\mathrm{Rep}(Q_3)_{(2,2)}=\Spec  
      \C [a_i,b_i,c_i,d_i, i=1,2,3]$. It is easy to see that
     \begin{align*}
         \{&\mathrm{det}(Y_1),\mathrm{det}(Y_2),\mathrm{det}(Y_3),a_1d_2+a_2d_1-b_1c_2-b_2c_1,\\
         &a_1d_3+a_3d_1-b_1c_3-b_3c_1,a_3d_2+a_2d_3-b_3c_2-b_2c_3\}
     \end{align*}
     are all $\Theta$-semi-invariant functions.

     Note the last three terms come from $\mathrm{det}(Y_i+Y_j)-\mathrm{det}(Y_i)-\mathrm{det}(Y_j)$. By \cite[Corollary 5.1]{Do}, the ring of semi-invariants  $\bigoplus_{r\geq 0}B(r\Theta)$ is the subring of $
      \C [a_i,b_i,c_i,d_i, i=1,2,3]$ generated by the above six elements. Moreover the subring generated by those six elements is a polynomial ring (\cite[Corollary 5.1 (i)]{Do}. Thus $M^{\Theta-ss}_{(2,2)}=\Proj\bigoplus_{r\geq 0}B(r\Theta)\cong \pP^5$.
\end{proof}
\begin{rem}
    From now on, we use $M_{rep}$ to denote $M^{\Theta-ss}_{(2,2)}$. We use 
    \begin{align*}
        \bigg(Y_1,Y_2,Y_3\bigg)
    \end{align*}
    to denote a representation where $Y_i$ corresponds to the arrow $a_i$.
\end{rem}

\subsection{Instanton Sheaves} Let $Y$ be a Fano threefold of index $2$. By definition an (minimal) \emph{instanton bundle} is a stable vector bundle $E$ of rank $2$ with Chern classes $c_1(E)=0$, $c_2(E)=2$, enjoying the instantonic condition
\begin{align*}
    H^1(Y,E(-1))=0.
\end{align*}

We use $M$ to denote the moduli space of semistable rank $2$ sheaves with Chern classes $c_1=0$, $c_2=2$ and $c_3=0$. It is clear that the moduli space of minimal instanton bundles $MI_2(Y)$ is an open subscheme of $M$. \cite[Theorem 4.7]{Ku2} gave a concrete description of $MI_2(V_5)$ using quadric nets. But it is not clear what $MI_2(V_5)$ looks like as a variety.\\

We also recall the definition of an Ulrich bundle.
\begin{defns}
    Let $X\subset \pP^N$ be a smooth projective variety of dimension $d$. An \emph{Ulrich bundle} $E$ is a vector bundle on $X$ satisfying 
    \begin{align*}
        H^*(X,E(-t))=0
    \end{align*}
    for all $t=1,\ldots,d$.
\end{defns}
We recall a few well-known facts about Ulrich bundles that will be useful to us:
\begin{itemize}
    \item There are no Ulrich line bundles on a variety $X$ with $\mathrm{Pic}(X)=\Z\cO_X(1)$.
    \item An Ulrich bundle is semistable. If it is not stable, it is an extension of Ulrich bundles of smaller ranks. 
\end{itemize}
One consequence of the above facts is that any rank $2$ Ulrich bundle on $V_5$ is stable.
To see the relation between (minimal) instanton bundles and Ulrich bundles of rank $2$ on $V_5$, we first recall the computation:
\begin{lem}\cite{Ku2}
    Let $E$ be an (minimal) instanton bundle on a Fano threefold of index $2$. Then 
    \begin{align*}
            H^*(E(t))=0
    \end{align*}
for $t=0,-1,-2$.
\end{lem}
On the other hand, Lee and Park obtained the following result in their recent paper:
\begin{prop}\cite[Proposition 3.4]{LP}
  For any $r\geq 2$, an Ulrich bundle $E$ of rank $r$ on $V_5$ corresponds to the following quiver representation.
    \begin{align*}
        E(-1)=\mathrm{Coker}(\mathcal{U}^{\oplus r}\to \mathcal{Q}\smvee^{\oplus r})
    \end{align*}
\end{prop}
Combine these two results, we immediately obtain:
\begin{cor}\label{Ulrich}
    A vector bundle $E$ on $V_5$ is an (minimal) instanton bundle if and only if $E(1)$ is an Ulrich bundle of rank 2.
\end{cor}
In addition, \cite{LP} described the moduli space of stable Ulrich bundles of rank $r$.
\begin{thm}\cite[Theorem 3.11]{LP}
    The moduli space $M^{sU}_r$ of stable Ulrich bundles of rank $r$ on $V_5$  is a smooth $(r^2+1)$-dimensional open subscheme of $M^{(-1,1)-s}_{(r,r)}(Q_3)$
\end{thm}
By the above corollary, we obtain the following relation when $r=2$,
\begin{align*}
    MI_2(V_5)\simeq M^{sU}_2\subset M
\end{align*}

\section{Classification of rank $2$ semistable sheaves on $V_5$}

We first look at stable rank $2$ vector bundles.
    
\begin{lem}\label{D22}
    Let $S\subset \pP^5$ be a del Pezzo surface of degree $5$ and $E$ a $\mu$-semistable vector bundle of rank $2$ with Chern classes $c_1(E)=0$ and $c_2(E)=2$. If $h^0(E)=0$, then $h^1(E(n))=0$  for $n\in \Z$ and $h^2(E(n))=0$ for $n\geq-1$. If $h^0(E)\neq 0$, then $h^0(E)=1$, $h^1(E(n))=0$ for $n\leq -2$ and $n\geq 1$, $h^1(E(-1))=h^1(E)=1$ and $h^2(E(n))=0$ for $n\geq 0$.
\end{lem}
\begin{proof}
    See \cite[Lemma 2.2]{D}. Suppose $h^0(E)=0$. Then we have $h^1(E)=h^0(E)=0$ since $h^2(E)=h^0(E(-1))=0$ and $\chi(E)=0$. Also $h^2(E(-1))=h^0(E)=0$. Now that we have $h^i(E(1-i))=0$ for $i\geq 1$, we can apply Proposition \ref{mumford} to $E(1)$ and obtain the first half of the lemma.\ 
    
    Suppose $h^0(E)\neq 0$. Then the zero locus of a nonzero global section is nonempty and of pure codimension $2$ by \cite{D} Lemma 2.1. We have an exact sequence by Proposition \ref{hlemma}:
    \begin{align*}
        0\to \cO_S\to E\to I_Z\to 0,
    \end{align*}
    where $Z$ is an closed subscheme of dimension $0$ and length $2$. We have $h^0(E)=1$ and $h^1(E)=1$ since $h^2(E)=h^0(E(-1))=0$ and $\chi(E)=0$. We claim the natural map $\alpha:H^0(\cO_S(1))\to H^0(\cO_Z(1))$ is surjective. Note $l(Z)=2$, if $Z$ consists of two distinct points $p_1,p_2$, one can find two hyperplanes of $\pP^5$ contained one of $p_i$ while avoiding the other, this shows $\alpha$ is surjective in this case. Similarly, if $Z$ consists of a closed point $q$ and a tangent direction at it, we can find a hyperplane not passing through $q$ and one hyperplane passing through $q$ while not containing the tangent direction. As a result of the surjectivity of $\alpha$, we have $h^1(I_Z(1))=0$. Then $h^1(E(1))=0$. Now that we have $h^i(E(2-i))=0$ for $i\geq 1$ and the second half of the lemma follows from apply Proposition \ref{mumford} to $E(2)$.
\end{proof}

\begin{lem}\label{D23}
    Let $S\subset \pP^5$ be a del Pezzo surface of degree $5$ and $E$ a $\mu$-semistable vector bundle of rank $2$ with Chern classes $c_1(E)=0$ and $c_2(E)=1$.  If $h^0(E)\neq 0$, then $h^0(E)=1$, $h^1(E(n))=0$ for $n\in\Z$ and and $h^2(E(n))=0$ for $n\geq 0$.
\end{lem}
\begin{proof}
    Use same arguments as in proof of second case of Lemma \ref{D22}.
\end{proof}

\begin{thm}\label{auto}
    Let $E$ be a stable rank $2$ vector bundle on $V_5$, with $c_1(E)=0$ and $c_2(E)=2$. Then $h^1(V_5,E(-1))=0$ or $1$
\end{thm}
\begin{proof}
    
    Let $S\in |\cO_{V_5}(1)|$ be a general hyperplane section of $V_5$. Then by \cite[Theorem 3.1]{M}, $E_S$ is $\mu$-semistable with respect to the polarization $\cO_S(1)$ .\\
    Suppose $h^0(E_S)=0$. Consider the short exact sequence:
    \begin{align*}
        0\to E(n-1)\to E(n)\to E_S(n)\to 0.
    \end{align*}
    Since $h^1(E_S(n))=0$ for $n\in \Z$, we have $h^1(E(n))\leq h^1(E(n-1))$. Thus $h^1(E(n))=0$ for all $n\in\Z$ since $h^1(E(n))=0$ for $n\ll0$.\\
    Suppose $h^0(E_S)\neq 0$. We claim $E(2)$ is generated by global sections. Using the same exact sequence above, we obtain $h^1(E(-n))=0$ for $n\geq 2$. Note $h^2(E)=h^1(E(-2))=0$ and $h^3(E)=h^0(E(-2))=0$. Since $\chi(E)=0$, we have $h^1(E)=0$ and the exact sequence:
    \begin{align}\label{C1}
        0\to E\to E(1)\to E_S(1)\to 0.
    \end{align}
    gives $h^1(E(1))=h^1(E_S(1))=0$. We have then $h^3(E(-1))=h^0(E(-1))=0$. Thus $E(2)$ is generated by global sections by Mumford-Castelnuovo criterion [Proposition \ref{mumford}].\\
    If there exist a section of $E(2)$ which vanishes nowhere then $E$ is isomorphic to $\cO_{V_5}(2)\oplus\cO_{V_5}(-2)$ and $c_2(E)=-20$, which is absurd. We have then an exact sequence:
    \begin{align}\label{C2}
        0\to \cO_{V_5}(-4)\to E(-2)\to I_C\to 0
    \end{align}
    where $C\subset V_5$ is a smooth (\cite[Proposition 1.4(b)]{H1}) curve of degree $c_2(E(2))=22$. We have $h^1(I_C)=0$, so the curve $C$ is connected. We have $\omega_C=\cO_C(2)$ (see Theorem \ref{Serre}) and $g(C)=23$. Finally, the curve $C$ is non-degenerate since $E$ is stable. The exact sequence:
    \begin{align*}
        0\to \cO_{V_5}(-3)\to E(-1)\to I_C(1)\to 0
    \end{align*}
    gives $h^1(I_C(1))=h^1(E(-1))$ since $h^1(\cO_{V_5}(-3))=0$ and $h^2(\cO_{V_5}(-3))=0$. The exact sequence:
    \begin{align*}
        0\to E(-2)\to E(-1)\to E_S(-1)\to 0
    \end{align*}
    gives $h^1(E(-1))=h^1(E_S(-1))=1$.
\end{proof}
\begin{rem}
    The argument above has been applied to prove same result for stable rank $2$ vector bundles on cubic threefolds \cite{D} and degree $4$ Fano threefold $V_4$. In fact, in those cases we can show $h^1(E(-1))=0$ by proving the curve $C$ constructed as above cannot exist. Unfortunately, we could not adapt those argument to prove the same thing for our case. We make the following conjecture:
\end{rem}
\begin{conj}\label{conj}
    On $V_5$, all stable rank $2$ vector bundles $E$ with Chern classes $c_1(E)=0$ and $c_2(E)=2$ satisfies the instantonic property: 
    \begin{align*}
        H^1(V_5,E(-1))=0.
    \end{align*}
    In particular, all such $E$ are minimal instanton bundles.
\end{conj}

\cite{D} classified rank $2$ semistable sheaves with Chern classes $c_1=0$, $c_2=2$ and $c_3=0$ on cubic threefolds. In this section we classify rank $2$ semistable sheaves with the same Chern classes on $V_5$, closely following the argument of \cite{D}.

\begin{prop}\label{d31}
  Let $E$ be a rank $2$ semistable sheaf with Chern classes $c_1=0$, $c_2=2$ and $c_3=0$ on $V_5$. Let $F$ be the double dual of $E$. Then either $E$ is locally free or $F$ is locally free with second Chern class $c_2(F)=1$ and $h^0(F)=1$ or $F=H^0(F)\otimes\cO_{V_5}$.
\end{prop}
\begin{proof}
    See \cite[Proposition 3.1]{D}. Let $S\in |\cO_{V_5}(1)|$ be a general hyperplane section so that $E_S$ is $\mu$-semistable with respect to the polarization $\cO_S(1)$ (\cite[Theorem 3.1]{M}) and so that $F_S$ is isomorphic to the double dual of $E_S$. The sheaf $F$ is $\mu$-semistable. The sheaf $F_S$ is locally free of rank $2$ and $\mu$-semistable with $c_1(F_S)=0$(\cite{H2}). Let $R$ denote the cokernel of the natural map $E\to F$. The sheaf $E$ is torsion free and $R$ is of dimension at most $1$. We have $c_2(F_S)=c_2(E_S)+c_2(R_S)=2-l(R_S)$ and $\chi(F_S)=l(R_S)$. We obtain the relation that $h^0(F_S)=h^1(F_S)+l(R_S)$ because $h^2(F_S)=h^0(F_S(-1))=0$. Note if $h^0(F_S)=0$ then $l(R_S)=0$. Suppose $h^0(F_S)\geq 1$. Then the zero locus of a non-zero section is either empty or of pure codimension $2$ (\cite[Lemma 2.1]{D}). If it is empty then $F_S$ is trivial and if codimension $2$ then $h^0(F_S)=1$. Thus in general we have $l(R_S)\in\{0,1,2\}$ and $c_2(F_S)\in\{0,1,2\}$.\
    
    Consider the exact sequence:
    \begin{align*}
        0\to F(n-1)\to F(n)\to F_S(n)\to 0.
    \end{align*}
    We have $h^1(F_S(n))=0$ for $n\leq -2$ and $n\geq 1$ (Lemma \ref{D22}, Lemma \ref{D23}). We then deduce $h^1(F(n-1))\geq h^1(F(n))$ for $n\leq -2$ and $n\geq 1$. Now since $h^1(F(n))=0$ for $n\ll0$(\cite[Theorem 2.5]{H2}) and $h^2(F(n))=0$ for $n\gg0$, we have $h^1(F(n))=0$ for all $n\leq -2$ and $h^2(F(n))=0$ for $n\geq 0$. We finally have $h^3(F)=h^0(F^{\vee}(-2))=0$ (\cite[Proposition 1.10]{H2}) and $h^0(F)\leq h^0(F_S)$.\
    
    Suppose $l(R_S)=0$. Then $c_2(F)=2$ and $\chi(F)=c_3(F)/2$. We then deduce the formula $c_3(F)/2=h^0(F)-h^1(F)$. Now $c_3(F)\geq 0$ (\cite[Proposition 2.6]{H2} ) and $h^0(F)\leq h^0(F_S)\leq 1$ and we then have $c_3(F)=0$ or $2$.\
    
    If $c_3(F)=0$, then the sheaves $E$ and $F$ are canonically isomorphic and locally free(\cite[Proposition 2.6]{H2}).\
    
    If $c_3(F)=2$, then $h^0(F)=1$ and $h^1(F)=0$. The sheaf $R$ is of dimension $0$ and $l(R)=\chi(F)-\chi(E)=1$. We thus have $R=\C(p)$ for some $p\in V_5$. Since $h^0(F)=1$, we have a non zero morphism $\cO_{V_5}\to F$. Furthermore,  since $\chi(E(n))=\frac{5}{3}n^3+5n^2+\frac{10}{3}n$ and $\chi(\cO_{V_5}(n))=\frac{5}{6}n^3+\frac{5}{2}n^2+\frac{8}{3}n+1$, we have $h^0(E)=0$ because $E$ is semistable. The induced morphism $\cO_{V_5}\to R$ is thus nonzero. It is surjective and induce an inclusion $I_p\subset V_5$. We have $\chi(I_p(n))=\frac{5}{6}n^3+\frac{5}{2}n^2+\frac{8}{3}n$, which is in contradiction with the semistability of $E$.\
    
    Suppose $l(R_S)\geq 1$. We have then $h^0(F_S)\geq 1$. The zero locus of a nonzero global section is
either empty or pure codimension 2.
Suppose there is a section of $F_S$ whose zero locus is pure codimension $2$. Then $h^0(F_S)=1$. We deduce $h^1(F_S)=0$ ,then $l(R_S)=1$ and $c_2(F)=1$. We have the inequality $\chi(F)=h^0(F)-h^1(F)=1+c_3(F)/2\leq 1-h^1(F)$. Then $c_3(F)=0$. Since $c_3(F)\geq 0$(\cite[Proposition 2.6]{H2}). Then sheaf $F$ is then locally free (\cite[Proposition 2.6]{H2}) with second Chern class $c_2(F)=1$ and $h^0(F)=1$.\

Suppose finally there is a non-vanishing section of $F_S$ in which case the restriction is isomorphic to $H^0(F_S)\otimes \cO_S$ and then $l(R_S)=2$ and $c_2(F)=0$. We obtain the inequality $\chi(F)=c_3(F)/2+2=h^0(F)-h^1(F)\leq 2-h^1(F)$, then $c_3(F)=0$ since $c_3(F)\geq 0$(\cite[Proposition 2.6]{H2}) and $h^0(F)=2$. The sheaf $F$ is thus locally free. Suppose there exist a nonzero section whose zero locus $Z$ is nonempty. The scheme $Z$ is pure of dimension $1$ since $h^0(F(-1))=0$ and $F$ is then the extension of $I_Z$ by $\cO_{V_5}$. We then have $h^0(F)=1$ which is absurd. Thus the sheaf $F$ is then isomorphic to $H^0(F)\otimes \cO_{V_5}$.
\end{proof}

\begin{lem}\label{d34}
    Suppose $\theta$ is  the theta-characteristic of a smooth conic $C\subset V_5$. We consider the sheaf $E$ which is the kernel of the surjection $H^0(\theta(1))\otimes\cO_{V_5}\to\theta(1)$. Then $E$ is stable with Chern classes $c_1(E)=0$, $c_2(E)=2$ and $c_3(E)=0$.
\end{lem}

\begin{proof}
     See \cite[Lemma 3.4]{D}.     The calculation of Chern class is immediate. Let $F\subset E$ be a subsheaf of rank $1$. Let $F\hookrightarrow F\smvee\smvee$ be the natural injection of $F$ into its double dual. Since $F\smvee\smvee$ is a reflexive sheaf of rank $1$ on a smooth variety, it is invertible.  Since $\mathrm{Pic}(V_5)\cong\Z$, $F\smvee\smvee=\cO_{V_5}(a)$ for some $a\in \Z$, we have injection $F(-a)\hookrightarrow \cO_{V_5}$.  Since $c_1(F)=c_1(F\smvee\smvee)=a$, the sheaf $F$ is of the form $I_Z(a)$ where $Z\subset V_5$ is a close subscheme of dimension at most $1$ and $a\in \Z$. We have a commutative diagram whose rows and columns are exact:\\
         \[\begin{tikzcd}
         &0\arrow{d}&0\arrow{d}\\
    &H^0(\theta(1))\otimes I_C\ar[equal]{r}\arrow{d}&H^0(\theta(1))\otimes I_C\arrow{d}\\
    0\arrow{r} &E\arrow{r}\arrow{d}&H^0(\theta(1))\otimes \cO_{V_5}\arrow{r}\arrow{d}&\theta(1)\arrow{r}\arrow[equal]{d}&0\\
    0\arrow{r}&\theta\arrow{r}\arrow{d}&H^0(\theta(1))\otimes \cO_{C}\arrow{r}\arrow{d}&\theta(1)\\
    &0&0
    \end{tikzcd}
    \]
    Denote $F_0$ the kernel of the induced map $F\to \theta$. We have an inclusion $F_0\subset H^0(\theta(1))\otimes I_C$. The sheaf $H^0(\theta(1))\otimes I_C$ is $\mu$-semistable of zero slope so we have $c_1(F)=c_1(F_0)\leq c_1(H^0(\theta(1))\otimes I_C)=0$ because $\theta$ is of dimension $1$.\
    
    If $c_1(F)<0$, we have $\chi(F(n))<\frac{1}{2}\chi(E(n))$ for $n\gg0$. If $c_1(F)=0$, then $F=I_Z$ with $\mathrm{codim}(Z)\geq 2$ and we have then $I^{\vee\vee}_Z=\cO_{V_5}$. The inclusion $I_Z\subset H^0(\theta(1))\otimes\cO_{V_5}$ induced from the inclusion $E\subset H^0(\theta(1))\otimes\cO_{V_5}$ is given by a nonzero section $s\in H^0(\theta(1))$. The induced map $I_Z\to \theta(1)$ therefore associates the section $f|C$ with the function $f$. The section is generally nonzero so $I_Z\subset I_C$ and then $\chi(I_Z(n))\leq\chi(I_C(n))<1/2\chi(E(n))$ for $n\gg0$ because $\chi(I_C(n))=\frac{5}{6}n^3+\frac{5}{2}n^2+\frac{2}{3}n$ and $\chi(E(n))=\frac{5}{3}n^3+5n^2+\frac{10}{3}n$.
\end{proof}
\begin{thm}
    Let $E$ be a rank $2$ semistable sheaf with Chern classes $c_1=0$, $c_2=2$ and $c_3=0$ on $V_5$. If $E$ is stable, then either $E$ is locally free or $E$ is associated to a smooth conic $C\subset V_5$ such that we have the exact sequence:
     \begin{align*}
         0\to E\to H^0(\theta(1))\otimes \cO_{V_5}\to \theta(1)\to 0
     \end{align*}
     where $\theta$ is the theta-characteristic of $C$.\\
     If $E$ is strictly semistable, then  $E$ is the extension of two ideal sheaves of lines in $V_5$.
\end{thm}
\begin{proof}
    See \cite[Theorem 3.5]{D}.  Let $F$ be the double dual of $E$ and $R$ the cokernel of the canonical injection $E\to F$. By Proposition \ref{d31}, either $E$ is locally free, or $F$ is locally free with second Chern class $c_2(F)=1$ and $h^0(F)=1$, or $F=H^0(F)\otimes \cO_{V_5}$. As before, $\chi(E(n))=\frac{5}{3}n^3+5n^2+\frac{10}{3}n$ and $\chi(\cO_{V_5}(n))=\frac{5}{6}n^3+\frac{5}{2}n^2+\frac{8}{3}n+1$. It follows that $h^0(E)=0$ since $E$ is semistable.\
    
    Suppose $E$ is locally free. Then $h^0(E)=0$ as above and the bundle $E$ is then stable.\
    
    Suppose $F$ is locally free with second Chern class $c_2(F)=1$ and $h^0(F)=1$. Then $\chi(R(n))=n+1$. Let $s\in H^0(F)$ be a non-trivial section. It vanishes along a line $l_2\in V_5$ by \cite[Proposition 2.1]{D}. We have $h^0(E)=0$ and the section $s$ of $F$ induce a nonzero map $\cO_{V_5}\to R$. Denote $I_Z$ the kernel of this map. We claim the scheme $Z$ is of dimension $1$. Otherwise we would have an inclusion $I_Z\subset E$ with $\chi(I_Z(n))=\frac{5}{6}n^3+\frac{5}{2}n^2+\frac{8}{3}n+1-l(Z)$, which is impossible by the semistability of $E$. Let $S\in |\cO_{V_5}(1)|$ be a general hyperplane section. The inclusion $\cO_{Z\cap S}\subset R_S$ is an isomorphism because $l(R_S)=1$. Also the support of $Z$ of dimension $1$ is then a line $l_1$ and we have a surjective morphism $\cO_Z\twoheadrightarrow \cO_{l_1}$. We then deduce $R=\cO_{l_1}$, because the two sheaves have the same characteristic polynomial. The morphism $\cO_{V_5}\to \cO_{l_1}$ is nontrivial and the lines $l_1$ and $l_2$ are then disjoint. Consider the exact commutative diagram:\
         \[\begin{tikzcd}
         &0\arrow{d}&0\arrow{d}\\
    0\ar{r}&I_{l_1}\ar{r}\arrow{d}&\cO_{V_5}\arrow{d}\ar{r}&\cO_{l_1}\ar[equal]{d}\ar{r}&0\\
    0\arrow{r} &E\arrow{r}\arrow{d}&F\arrow{r}\arrow{d}&\cO_{l_1}\arrow{r}&0\\
    &I_{l_2}\arrow[equal]{r}\arrow{d}&I_{l_2}\arrow{d}\\
    &0&0
    \end{tikzcd}
    \]
    We can see $E$ is strictly semistable and is an extension of $I_{l_1}$ and $I_{l_2}$.\
    
    Suppose $F=H^0(F)\otimes \cO_{V_5}$. We have $\chi(R(n))=2n+2$ and $h^0(E)=0$. Then in particular $h^0(R)\geq 2$. Let $S\in |\cO_{V_5}(1)|$ be a general hyperplane section. We have an exact sequence:
    \begin{align*}
        0\to E_S\to H^0(F)\otimes\cO_S\to R_S\to 0.
    \end{align*}
    The map $H^0(H^0(F)\otimes\cO_S)\to H^0(R_S)$ is not zero since the morphism $H^0(F)\otimes\cO_S\to R_S$ is surjective. We have then $h^0(E_S)\leq 1$. If $h^0(E_S)=0$, then the map $H^0(F_S)\to H^0(R_S)$ is an isomorphism and the map $H^0(F_S)\otimes H^0(\cO_S(n))\to H^0(R_S(n))$ is surjective for all $n\geq 0$ since there exist a hyperplane section of $S$ avoiding the support of $R_S$. If $h^0(E_S)=1$ then the quotient $H^0(F_S)/H^0(E_S)$ is of dimension $1$ and we have a surjective map $(H^0(F_S)/H^0(E_S))\otimes \cO_S\to R_S$. Now $l(R_S)=2$ and $\cO_S(1)$ is very ample and the map $(H^0(F_S)/H^0(E_S))\otimes \cO_S(n)\to H^0(R_S(n))$ is the surjective for $n\geq 1$. As a result, the map $H^0(F_S)\otimes H^0(\cO_S(n))\to H^0(R_S(n))$ is also surjective. We have finally $h^1(E_S(n))=0$ for $n\geq 1$ because $h^1(\cO_S(n))=0$ for $n\geq 1$. The exact sequence
    \begin{align*}
        0\to E(n-1)\to E(n)\to E_S(n)\to 0
    \end{align*}
    gives $h^2(E(n-1)\leq h^2(E(n))$ for $n\geq 1$. We have then $h^2(E(n))=0$ for $n\geq 0$ because $h^2(E(n))=0$ for $n\gg0$. In particular $h^2(E)=0$. We obtain the exact sequence
    \begin{align*}
        0\to E\to H^0(F)\otimes \cO_{V_5}\to R\to 0
    \end{align*}
    and also $h^3(E)=0$. But $\chi(E)=0$, thus  we have $h^1(E)=0$. We deduce $h^0(R)=2$ and the inclusion $H^0(F)\to H^0(R)$ is an isomorphism. We claim the restriction map $H^0(R)\to H^0(R_S)$ is injective. Suppose otherwise. Then there exist a nonzero section $s\in H^0(R)$  whose image in $H^0(R_S)$ is zero. Denote by $I_Z$ the kernel of the map $\cO_{V_5}\to R$ defined by the section $s$ and $Q$ the cokernel of the inclusion $\cO_Z\to R$. By hypothesis, the map $\cO_{Z\cap S}\to R_S$ is zero and we have then $R_S=Q_S$. The sheaf $Q$ is of dimension $1$ and $c_2(Q)=c_2(R)$. The scheme $Z$ is then of dimension $0$. Note $\chi(I_Z(n))=\frac{5}{6}n^3+\frac{5}{2}n^2+\frac{8}{3}n+1-l(Z)$ which is in contradiction with the semistability of $E$ since we have the inclusion $I_Z\subset E$. The restriction map $H^0(R)\to H^0(R_S)$ is injective and we have then $h^0(R(-1))=0$. Then by \cite[Lemma 3.3]{D} the sheaf $R(-1)$ is either a theta characteristic of a smooth conic $C\subset V_5$ in which case $E$ is stable by Lemma \ref{d34}, or there exists two lines $l_1,l_2\subset V_5$ so that $R$ is an extension of $\cO_{l_1}$ and $\cO_{l_2}$ in which case we have the commutative diagram:
    \[\begin{tikzcd}
    &0\arrow{d}&0\arrow{d}&0\ar{d}\\
    0\ar{r}&I_{l_1}\ar{r}\arrow{d}&E\arrow{d}\ar{r}&I_{l_2}\ar{d}\ar{r}&0\\
    0\arrow{r} &\cO_{V_5}\arrow{r}\arrow{d}&H^0(F)\otimes\cO_{V_5}\arrow{r}\arrow{d}&\cO_{V_5}\ar{d}\arrow{r}&0\\
    0\ar{r}&\cO_{l_1}\arrow{r}\arrow{d}&R\ar{r}\arrow{d}&\cO_{l_2}\ar{d}\ar{r}&0\\
    &0&0&0
    \end{tikzcd}
    \]
    Then $E$ is an extension of $I_{l_1}$ and $I_{l_2}$.
\end{proof}

\section{Relation to representation of $Q_3$}
Given any representation $R$ of the quiver $Q_3$, we can obtain via $\Psi^{-1}$ a complex 
\begin{align*}
   C_R: \mathcal{U}^{\oplus 2}\xrightarrow{f_R}\mathcal{Q}\smvee^{\oplus 2}
\end{align*}
in $\mathcal{D}^b(V_5)$ where we put the second term in degree $0$.
\begin{prop}\label{p1}
  If $R$ is semistable, then the corresponding map $f_R:\mathcal{U}^{\oplus 2}\to \mathcal{Q}\smvee^{\oplus 2}$
  is injective.
\end{prop}
\begin{proof}
    We use $\mathcal{K}$ and $\mathcal{I}$ to denote the kernel and image of $f_R$. Note both $\ke$ and $\im$ are torsion-free. Since $R$ is semistable, $f_R$ is nonzero. So the rank of $\im$ is $1,2,3$ or $4$. Note if $\rank(\im)=4$, then $\ker$ is torsion-free of rank $0$, hence $0$ and we are done. It remains to exclude the other cases.\
    
    Since $\mathcal{U}^{\oplus 2}$ and $\mathcal{Q}\smvee^{\oplus 2}$ are $\mu$-semistable with slope $-1/2$ and $-1/3$, we have $-1/2\leq\mu(\im)\leq-1/3$. Thus $\rank(\im)$ is either $2$ or $3$.\ 
    
    If $\rank(\im)=2$, then $c_1(\im)=-1$, $\rank(\ke)=2$ and $c_1(\ke)=-1$. Since $\mu(\ke)=\mu(\mathcal{U}^{\oplus 2})$, we know $\ke$ is $\mu$-semistable. In turn, $\im$ is also $\mu$-semistable. Since there is no rank $1$ sheaf with $\mu=-1/2$, $\im$ and $\ke$ are $\mu$-stable. By the uniqueness of Jordan-Holder filtration, $\ke\cong \im \cong \mathcal{U}$. Thus $R$ has a subrepresentation of dimension $(1,0)$, which contradicts its semistability.\  
    
    If $\rank(\im)=3$, then $c_1(\im)=-1$. Recall we use $\mathcal{C}$ to denote the cokernel of $f_R$ and let $\mathcal{C}_{tf}$ to be the quotient of $\mathcal{C}$ by its torsion part. Let $\im'$ be the kernel of the map $\mathcal{Q}\smvee^{\oplus 2}\to \mathcal{C}_{tf}$. We then have $\im\hookrightarrow\im'\hookrightarrow \mathcal{Q}\smvee^{\oplus 2}$. Now $\rank(\im')=\rank(\im)=3$ and $c_1(\im')\geq c_1(\im)=-1$. By the semistability of $\mathcal{Q}\smvee^{\oplus 2}$, $c_1(\im')=-1$. Thus $\mathcal{C}_{tf}$ is a torsion-free sheaf of rank $3$ and $c_1=-1$. By the same argument as in the previous paragraph, we can show $\im'\cong\mathcal{C}_{tf}\cong \mathcal{Q}\smvee$. Thus $R$ has a subrepresentation of dimension $(2,1)$, which contradicts its semistability.\  
\end{proof}
We now know if $R$ is semistable, $C_R$ is isomorphic in $\mathcal{D}^b(V_5)$ to a sheaf $E$. A straight forward computation shows $\rank(E)=2$, $c_1(E)=0$, $c_2(E)=2$ and $c_3(E)=0$.
\begin{prop}\label{p2}
  With the notation above, $E$ is a semistable sheaf. 
\end{prop}
\begin{proof}
    We first show $E$ is torsion-free. Suppose $E_{tor}\neq 0$. If $c_1(E_{tor})>0$, Then $\mu(E_{tf})\leq -1/2$, which contradicts the semistability of $\mathcal{Q}\smvee^{\oplus 2}$. Thus $E_{tor}$ is supported in dimension at most $1$. Now let $K$ be the kernel of the surjection $\mathcal{Q}\smvee^{\oplus 2}\to E_{tf}$. We have a short exact sequence:
    \begin{align*}
        0\to \mathcal{U}^{\oplus2}\to K\to E_{tor}\to 0.
    \end{align*}
    Note we have
    \begin{align*}
        \Ext^1(E_{tor},\mathcal{U}^{\oplus2})&=\Ext^2(\mathcal{U}^{\oplus2},E_{tor}\otimes \cO_{V_5}(-2))\\
        &=H^2(E_{tor}\otimes\cO_{V_5}(-2)\otimes (\mathcal{U}^{\oplus2})\smvee)\\
        &=0
    \end{align*}
    since $E_{tor}$ is supported in dimension at most $1$. Thus $K\cong \mathcal{U}^{\oplus2}\oplus E_{tor}$. This contradicts that fact that $K$ is torsion-free.\
    
    It remains to show that any rank $1$ subsheaf $E'$ of $E$ satisfies $p(E')\leq p(E)$. If $c_1(E')<0$, the above inequality holds. If $c_1(E')>0$, let $E''=E/E'$, then the kernel $J$ of the map $\mathcal{Q}\smvee^{\oplus 2}\to E''$ is given by the short exact sequence:
    \begin{align*}
        0\to \mathcal{U}^{\oplus2}\to J\to E'\to 0
    \end{align*}
    Then $\mu(J)\geq -1/5$. This is absurd since $J$ is a subsheaf of $\mathcal{Q}\smvee^{\oplus 2}$. If $c_1(E')=0$, then $E'$ is the ideal sheaf of a closed subscheme of dimension at most $1$. We have $c_2(E')$ to be $0,1$ or $2$. If $c_2(E')=2$, then $p(E')<p(E)$. If $c_2(E')=0$, $E'$ is the ideal sheaf of a subscheme $Z$ supported in dimension $0$. Using the above notation we see $J$ is again an extension of $E'$ and $\mathcal{U}^{\oplus2}$, we have exact sequence:
    \begin{align*}
        \Ext^1(\cO_{V_5},\mathcal{U}^{\oplus2})\to\Ext^1(E',\mathcal{U}^{\oplus2})\to \Ext^2(\cO_Z,\mathcal{U}^{\oplus2})
    \end{align*}
    It is straight forward to see the left and right terms are $0$. Thus $J=E'\oplus\mathcal{U}^{\oplus2}$. Then $E'$ is a subsheaf of $\mathcal{Q}\smvee^{\oplus 2}$, which contradicts its semistability. If $c_2(E')=1$. By looking at the Euler characteristic, we see $ch_3(E')\leq 0$. When $ch_3(E')<0$, we have $p(E')<p(E)$. When $ch_3(E')=0$, $p(E')=p(E)$ and $E'$ is the ideal sheaf of a line.
\end{proof}
\begin{cor}\label{aaa}
    If $R$ is stable (strictly semistable), then the corresponding $E$ is stable (strictly semistable).
\end{cor}
\begin{proof}
    Suppose $R$ is strictly semistable, then $R$ has a subrepresentation of dimension $(1,1)$. Note all nonzero morphisms from $\mathcal{U}$ to $\mathcal{Q}\smvee$ are injective and has $I_l$ as cokernel for some line $l$ in $V_5$.  Thus we have the commutative diagram:
    \[\begin{tikzcd}
    0\arrow{r}\arrow{d} &\mathcal{U}\arrow{r}\arrow{d}&\mathcal{Q}\smvee\arrow{r}\arrow{d}&I_{l}\arrow{r}\arrow{d}&0\\
    0\arrow{r}\arrow{d} &\mathcal{U}^{\oplus 2}\arrow{r}\arrow{d}&\mathcal{Q}\smvee^{\oplus 2}\arrow{r}\arrow{d}&E\arrow{r}\arrow{d}&0\\
    Ker(f)\arrow{r}&\mathcal{U}\arrow{r}&\mathcal{Q}\smvee\arrow{r}&Coker(f)
    \end{tikzcd}
    \]
    Again by the injectivity of nonzero morphisms from $\mathcal{U}$ to $\mathcal{Q}\smvee$, 
    we have either $Ker(f)=0$ or $Ker(f)=\mathcal{U}$. But the second case would imply there is an injection $\mathcal{U}^{\oplus 2}\to \mathcal{Q}\smvee$, which makes $R$ unstable. Thus $Ker(f)=0$ and $Coker(f)=I_{l'}$ for some line $l'$. By the snake lemma, we see $I_{l}\hookrightarrow E$, so $E$ is strictly semistable.
    
    If $E$ is strictly semistable, we have short exact sequence:
    \begin{align*}
        0\to I_{l_1}\to E\to I_{l_2}\to 0,
    \end{align*}
    where $l_1,l_2$ are two lines in $V_5$. By \cite[Lemma 4.2]{Ku2}, we have exact sequence:
    \begin{align*}
        0\to \mathcal{U}\to \mathcal{Q}\smvee\to I_{l_1}\to 0.
    \end{align*}
    Since $\{\mathcal{U},\mathcal{Q}\smvee\}$ is an exceptional pair, $\Hom(\mathcal{Q}\smvee,\mathcal{Q}\smvee^{\oplus 2})=\Hom(\mathcal{Q}\smvee,E)$, we obtain the commutative diagram:
    \[\begin{tikzcd}
    &0\arrow{d}&0\arrow{d}&0\arrow{d}&\\
    0\arrow{r} &\mathcal{U}\arrow{r}\arrow{d}&\mathcal{Q}\smvee\arrow{r}\arrow{d}&I_{l_1}\arrow{r}\arrow{d}&0\\
    0\arrow{r} &\mathcal{U}^{\oplus 2}\arrow{r}&\mathcal{Q}\smvee^{\oplus 2}\arrow{r}&E\arrow{r}&0\\
    \end{tikzcd}
    \]
    This shows $R$ has a subrepresentation with dimension vector $(1,1)$, thus it is not stable.
\end{proof}
\cite{LP} proved that for any Ulrich bundle $E$ of rank $r\geq 2$ on $V_5$, $E(-1)$ is a semistable representation of $Q_3$ with dimension vector $(r,r)$. By Corollary \ref{Ulrich}, this implies minimal instanton bundles are representations of $Q_3$ with dimension vector $(2,2)$.
\begin{thm}[\cite{LP} Proposition 3.4, 3.10]\label{rep}
    Let $E$ be a minimal instanton bundle on $V_5$, then there exists a short exact sequence:
    \begin{align*}
        0\to \mathcal{U}^{\oplus 2}\to \mathcal{Q}\smvee^{\oplus 2}\to E\to 0.
    \end{align*}
    In particular, $E\in\mathcal{B}_{V_5}$ and $\Psi(E)$ is isomorphic to a stable representation of $Q_3$ with dimension vector $(2,2)$.
\end{thm}
\begin{rem}
    If $E$ is a stable rank $2$ bundle with Chern classes $c_1(E)=0$, $c_2(E)=2$, while satisfying $H^1(E(-1))\neq 0$, then $E\notin \mathcal{B}_{V_5}$ and we cannot associate to $E$ any representation of $Q_3$.
\end{rem}
The following two lemmas will help us to generalize Theorem \ref{rep}.
\begin{lem}\label{12344}
    Let $E$ be a semistable rank $2$ sheaf with Chern classes $c_1(E)=0$, $c_2(E)=2$ and $c_3(E)=0$ on $V_5$ that is not locally free. Then $E\in \mathcal{B}_{V_5}$.
\end{lem}
\begin{proof}
    It suffices to show that $H^*(E(-1))=H^*(E)=0$. If $E$ is associated to a smooth conic, we have short exact sequence:
    \begin{align*}
        0\to E\to H^0(\theta(1))\otimes\cO_{V_5}\to \theta(1)\to 0
    \end{align*}
    Since $H^*(\cO_{V_5}(-1))=H^*(\theta)=0$, we immediately obtain $H^*(E(-1))=0$. On the other hand, we have $h^0(H^0(\theta(1))\otimes\cO_{V_5})=2$ and $h^0(\theta(1))=2$. It is clear that the map $H^0(H^0(\theta(1))\otimes\cO_{V_5}) \to H^0(\theta(1))$ is surjective. Moreover, $H^i(\cO_{V_5})=H^i(\theta(1))=0$ for all $i>0$. Thus $H^*(E)=0$.\\
    If $E$ is the extension of the ideal sheaves of lines in $V_5$, we note by \cite[Lemma 4.2]{Ku2}, $I_l\in \mathcal{B}_{V_5}$ for all lines $l$. The lemma follows immediately.
\end{proof}

\begin{rem}
    We have seen that all sheaves in Lemma \ref{12344} satisfy the property $H^1(E(-1))=0$. We make the following definition:
\end{rem}
\begin{defn}
    A sheaf $E$ on $V_5$ is called an \it{instanton sheaf} if $E$ is semistable of rank $2$ with Chern classes $c_1(E)=0$, $c_2(E)=2$ and $c_3(E)=0$, satisfying the condition:
    \begin{align*}
        H^1(E(-1))=0.
    \end{align*}
\end{defn}
If a semistable sheaf $F$ of rank $2$ with Chern classes $c_1(F)=0$, $c_2(F)=2$ and $c_3(F)=0$ is not instanton, then by Theorem \ref{auto} and Lemma \ref{12344}, $F$ is a stable bundle with $H^1(F(-1))\cong \C$. We mention again that it is our conjecture that such $F$ does not exist.

\begin{lem}\label{surj}
    Let $C$ be a smooth conic on $V_5$. Consider the short exact sequence:
    \begin{align*}
        0\to \mathcal{Q}\otimes I_C\to \mathcal{Q}\to \mathcal{Q}|_C\to 0
    \end{align*}
    The induced map
    \begin{align*}
        f\colon H^0(\mathcal{Q})\to H^0(\mathcal{Q}|_C)
    \end{align*}
    is an isomorphism.
\end{lem}
\begin{proof}
    We have exact sequence:
    \begin{align*}
        0\to H^0(\mathcal{Q}\otimes I_C)\to H^0(\mathcal{Q})\xrightarrow[]{f} H^0(\mathcal{Q}|_C)\to H^1(\mathcal{Q}\otimes I_C)
    \end{align*}
    Note $h^0(\mathcal{Q})=h^3(\mathcal{Q}\smvee(-2))=5$ and $h^0(\mathcal{Q}|_C)=h^0(\theta(1)\oplus\theta(1)\oplus\cO_C)=5$. To show $f$ is an isomorphism, it suffices to show $H^1(\mathcal{Q}\otimes I_C)=0$. Consider the short exact sequence:
    \begin{align*}
        0\to U\to V\otimes \cO_{V_5}\to \mathcal{Q}\to 0
    \end{align*}
    we have
    \begin{align*}
        H^1(I_C)^{\oplus 5}\to H^1(\mathcal{Q}\otimes I_C)\to H^2(\mathcal{U}\otimes I_C)
    \end{align*}
    It is clear $H^1(I_C)=0$, so it remains to show $H^2(\mathcal{U}\otimes I_C)=0$. Consider the short exact sequence:
    \begin{align*}
        0\to \mathcal{U}\otimes I_C\to \mathcal{U} \to \mathcal{U}|_C\to 0
    \end{align*}
    we have 
    \begin{align*}
        H^1(\mathcal{U}|_C)\to H^2(\mathcal{U}\otimes I_C)\to H^2(\mathcal{U})
    \end{align*}
    But $H^2(\mathcal{U})=0$ and $H^1(\mathcal{U}|_C)=H^1(C,\theta\oplus \theta)=0$, thus $H^2(\mathcal{U}\otimes I_C)=0$ and the proof is complete.
\end{proof}
We generalize Theorem \ref{rep} slightly:
\begin{thm}\label{rep2}
    Let $E$ be an instanton sheaf on $V_5$, then there exists a short exact sequence:
    \begin{align*}
        0\to \mathcal{U}^{\oplus 2}\to \mathcal{Q}\smvee^{\oplus 2}\to E\to 0.
    \end{align*}
    In particular, $E\in\mathcal{B}_{V_5}$ and $\Psi(E)$ is isomorphic to a  representation of $Q_3$ with dimension vector $(2,2)$.
\end{thm}
\begin{proof}
    By Lemma \ref{repcompute}, it suffices to show 
   $\Ext^\bullet(E,\mathcal{U}[1])=\C^2$ and 
    $\Ext^\bullet(\mathcal{Q}\smvee,E)=\C^2$.\\
    If $E$ is a vector bundle, the result is Theorem \ref{rep}.\\
    If $E$ is associated to a smooth conic, we have the short exact sequence:
    \begin{align*}
        0\to E\to H^0(\theta(1))\otimes\cO_{V_5}\to \theta(1)\to 0.
    \end{align*}
    Taking the long exact sequence and noting $H^*(\mathcal{U})=0$, we have
    \begin{align*}
        \Ext^i(E,\mathcal{U})=\Ext^{i+1}(\theta(1),\mathcal{U})
    \end{align*}
    It is easy to see the only nonzero part is 
    \begin{align*}
        \Ext^1(E,\mathcal{U})&=\Ext^{2}(\theta(1),\mathcal{U})\\
        &=\Ext^1(\mathcal{U}\smvee\otimes\theta(-1))\\
        &=\Ext_C^1(\cO_C(-1)\oplus\cO_C(-1))\\
        &=\C^2
    \end{align*}
    This proves $\Ext^\bullet(E,\mathcal{U}[1])=\C^2$. We now compute $\Ext^\bullet(\mathcal{Q}\smvee,E)$. 
    We have $\mathcal{Q}\otimes \theta(1)= \theta(1)\oplus\cO_C(1)\oplus\cO_C(1)$ and $H^k(\mathcal{Q}\otimes \theta(1))=0$ unless $k=0$. Moreover, $H^k(\mathcal{Q})=0$ unless $k=0$. Thus $\Ext^k(\mathcal{Q}\smvee,E)=0$ for $k\geq 2$ and we have exact sequence:
    \begin{align*}
        0\to \Hom(\mathcal{Q}\smvee,E)\to H^0(\theta(1))\otimes H^0(\mathcal{Q})\xrightarrow[]{g} H^0(\mathcal{Q}\otimes\theta(1))\to\Ext^1(\mathcal{Q}\smvee,E)\to 0
    \end{align*}
    Since $C\simeq \pP^1$, $g$ factors as
    \begin{align*}
        H^0(\theta(1))\otimes H^0(\mathcal{Q})\xrightarrow[]{id\otimes f} H^0(\theta(1))\otimes H^0(\mathcal{Q}|_C)\twoheadrightarrow H^0(\mathcal{Q}\otimes\theta(1))
    \end{align*}
    where $f$ is as in Lemma \ref{surj}, so $g$ is surjective. Thus $\Ext^1(\mathcal{Q}\smvee,E)=0$ and $\Hom(\mathcal{Q}\smvee,E)=\C^2$ by counting dimension.

    If $E$ is the extension of ideal sheaves of lines in $V_5$, then we have short exact sequence:
    \begin{align}\label{1}
        0\to I_{l_1}\to E\to I_{l_2}\to 0
    \end{align}
    By \cite[Lemma 4.2]{Ku2}, we have
    \begin{align*}
        &\Ext^\bullet(I_{l_k},\mathcal{U}[1])=\C\\ 
    &\Ext^\bullet(\mathcal{Q}\smvee,I_{l_k})=\C
    \end{align*}
    hence the result follows by looking at the long exact sequence induced by (\ref{1}).
\end{proof}

We now try to show all representations in Theorem \ref{rep2} are semistable.
\begin{prop}\label{ss}
  Let $E$ be a stable(strictly semistable) instanton sheaf on $V_5$. Then the corresponding representation in Theorem \ref{rep2} is $(-1,1)$-stable(strictly semistable).
\end{prop}
\begin{proof}
    Let $E$ be an instanton sheaf. By Theorem \ref{rep2}, there exists a short exact sequence:
    \begin{align*}
        0\to \mathcal{U}^{\oplus 2}\to \mathcal{Q}\smvee^{\oplus 2}\to E\to 0.
    \end{align*}
    We claim the induced representation is semistable. Suppose otherwise, then the representation has a subrepresentation of dimension either $(2,1)$ or $(1,0)$. If we are in the first case, then we have the following exact triangle in $\mathcal{D}^b(V_5)$:
    \begin{align*}
        A^{\bullet} \to E[0]\to B^{\bullet}
    \end{align*}
    where $A^\bullet\simeq (\mathcal{U}^{\oplus 2}\to \mathcal{Q}\smvee)$ and $B^\bullet\simeq \mathcal{Q}\smvee$ . Note in both $A^\bullet$ and $B^\bullet$, $Q\smvee$ lies in degree $0$. Taking the long exact sequence of cohomology sheaves of the exact triangle, we obtain:
    \begin{align*}
        0=\cH^{-2}(B^\bullet)\to \cH^{-1}(A^\bullet)\to \cH^{-1}(E[0])=0
    \end{align*}
    Thus $\cH^{-1}(A^\bullet)=0$, this would imply $\mathcal{U}^{\oplus 2}\hookrightarrow \mathcal{Q}\smvee$, which is absurd. The second case can be dealt with similarly. Thus the representation corresponding to $E$ is semistable.\\
    
    The rest follows from Corollary \ref{aaa}.
\end{proof}
By now we have established a well-behaved correspondence between (semi)stable instanton sheaves on $V_5$ and $(-1,1)$-(semi)stable representations of $Q_3$ with dimension vector $(2,2)$. Next is to use this correspondence to analyze the two moduli spaces.

\section{Moduli Space of Instantons}
Recall $\mathrm{Rep(Q_3)}^{(-1,1)-ss}_{(2,2)}$ denotes the quasi-affine scheme parametrizing $(-1,1)$-semistable representations of $Q_3$ with dimension vector $(2,2)$, $M_{rep}=\mathrm{Rep(Q_3)}^{(-1,1)-ss}_{(2,2)}//G_{2,2}$ is the moduli space of such representations, where $G_{2,2}=(\mathrm{GL}(2)\times \mathrm{GL}(2))/\C^*_{\mathrm{diag}}$. $M$ denotes the moduli space of semistable rank $2$ sheaves with Chern classes $c_1=0$, $c_2=2$ and $c_3=0$.\ 

By Proposition \ref{p1} and \ref{p2}, there exist a morphism $\Phi:\mathrm{Rep(Q_3)}^{(-1,1)-ss}_{(2,2)}\to M$. $\Phi$ is invariant under $G_{2,2}$ since changing the basis of $\mathcal{U}^{\oplus 2}$ and $\mathcal{Q}\smvee^{\oplus2}$ does not affect the isomorphism class of the quotient. Thus we obtain a morphism:
\begin{align*}
    \phi:M_{rep}\to M
\end{align*}
\begin{prop}\label{injective}
  $\phi$ is injective.
\end{prop}
\begin{proof}
    Let $R_1,R_2$ be two representations such that $\phi([R_1])=\phi([R_2])$.
If $\phi([R_1])$ is stable, then $E_{R_1}\cong E_{R_2}$.  Using the equivalence $\Psi$, we see $R_1\simeq R_2$ as objects in $\mathcal{D}^b(Q_3)$. But since they both lie in $\mathrm{Mod}(Q_3)$, we see they are actually isomorphic. Hence $[R_1]=[R_2]$.\ 

If $\phi([R_1])$ is strictly semistable. Then $E_{R_1}$, $E_{R_2}$ are both $S$-equivalent to $I_{l_1}\oplus I_{l_2}$ where $l_1,l_2$ are two lines in $V_5$. Then the proposition follows from the following lemma.
\end{proof}

\begin{lem}
    Let $E$ be a sheaf defined by
    \begin{align*}
        0\to I_{l_1}\to E\to I_{l_2}\to 0
    \end{align*}
    where $l_1,l_2$ are two lines in $V_5$. Then $\Psi(E)$ and $\Psi(I_{l_1}\oplus I_{l_2})$ corresponds to the same point in $M_{rep}$.
\end{lem}
\begin{proof}
    Recall $F(V_5)=\pP(A)\simeq \pP^2$. Suppose $l_1$ corresponds to $[a_1:b_1:c_1]$ and $l_2$ corresponds to $[a_2:b_2:c_2]$. Then $\Psi(E)$ is given by
    \[
     \bigg(
    \begin{bmatrix}
   
    a_1 &\alpha\\
    0   &a_2
    \end{bmatrix}
    \begin{bmatrix}
    b_1 &\beta\\
    0   &b_2
    \end{bmatrix}
    \begin{bmatrix}
    c_1 &\gamma\\
    0   &c_2
    \end{bmatrix}
    \bigg)
    \]
    where $\alpha,\beta,\gamma\in\C$. Let 
    \begin{align*}
        g_n=
        \bigg(
        \begin{bmatrix}
        \frac{1}{n}&0\\
        0&1
        \end{bmatrix},
        \begin{bmatrix}
        \frac{1}{n}&0\\
        0&1
        \end{bmatrix}
        \bigg)\in G_{2,2}
    \end{align*}
    then
    \begin{align*}
        g_n\cdot \Psi(E)=
        \bigg(
        \begin{bmatrix}
    a_1 &\frac{\alpha}{n}\\
    0   &a_2
    \end{bmatrix}
    \begin{bmatrix}
    b_1 &\frac{\beta}{n}\\
    0   &b_2
    \end{bmatrix}
    \begin{bmatrix}
    c_1 &\frac{\gamma}{n}\\
    0   &c_2
    \end{bmatrix}
    \bigg)
    \end{align*}
    and
    \begin{align*}
        \lim_{n\to\infty}g_n\cdot \Psi(E)&=
        \bigg(
        \begin{bmatrix}
        a_1 &0\\
    0   &a_2
    \end{bmatrix}
    \begin{bmatrix}
    b_1 &0\\
    0   &b_2
    \end{bmatrix}
    \begin{bmatrix}
    c_1 &0\\
    0   &c_2
    \end{bmatrix}
    \bigg)\\
    &=\Psi(I_{l_1}\oplus I_{l_2})
    \end{align*}
\end{proof}
We now try to understand the image of $\phi$. By semi-continuity, there exist an open subscheme $M^{inst}$ of $M$ parametrizing instanton sheaves.
We start by showing the smoothness of $M^{inst}$. To do this we first compute some related invariants.
\begin{lem}\label{smooth1}
    Let $\theta$ be the theta characteristic of a smooth conic $C$ in $V_5$. Let $E$ be the kernel of the natural surjection $H^0(\theta(1))\otimes\cO_{V_5}\to \theta(1)$. Then $\Ext^2(E,E)=0$ and $\Ext^1(E,E)$ has dimension $5$.
\end{lem}
\begin{proof}
    Consider the exact sequence:
    \begin{align*}
        \Ext^2(H^0(\theta(1))\otimes\cO_{V_5},E)\to\Ext^2(E,E)\to \Ext^3(\theta(1),E)
    \end{align*}
    We have $\Ext^2(H^0(\theta(1))\otimes\cO_{V_5},E)\simeq H^0(\theta(1))\otimes H^2(E)=0$ and $\Ext^3(\theta(1),E)\simeq \Hom(E,\theta(-1))^*$. Note we have injection $\Hom(E,\theta(-1))\to \Hom(\mathcal{Q}\smvee\oplus \mathcal{Q}\smvee,\theta(-1))=0$ by looking at the splitting type of $\mathcal{Q}\smvee$. Thus $\Ext^2(E,E)=0$. Moreover, $\Ext^3(E,E)\simeq\Hom(E,E(-2))^*=0$ and $\Hom(E,E)=\C$. By Riemann-Roch, $\chi(E,E)=-4$. Thus $\Ext^1(E,E)$ is five dimensional.
\end{proof}

\begin{lem}\label{smooth 2}
    Let $l_1,l_2\subset V_5$ be two lines. Then $\Ext^2(I_{l_1},I_{l_2})=0$ and $\dim\Ext^1(I_{l_1},I_{l_2})=1$ if $l_1\neq l_2$ and $2$ if $l_1=l_2$.
\end{lem}
\begin{proof}
    Consider the exact sequence:
    \begin{align*}
        \Ext^2(\cO_{V_5},I_{l_2})\to\Ext^2(I_{l_1},I_{l_2})\to \Ext^3(\cO_{l_1},I_{l_2})
    \end{align*}
    We have $\Ext^2(\cO_{V_5},I_{l_2})=0$ and $\Ext^3(\cO_{l_1},I_{l_2})\simeq \Hom(I_{l_2},\cO_{l_1}(-2))^*$. Note we have injection $$\Hom(I_{l_2},\cO_{l_1}(-2))\to\Hom(\mathcal{Q}\smvee,\cO_{l_1}(-2))=0$$ by looking at the splitting type of $\mathcal{Q}\smvee$. Thus $\Ext^2(I_{l_1},I_{l_2})=0$. Moreover, $\Ext^3(I_{l_1},I_{l_2})\simeq\Hom(I_{l_2},I_{l_1}(-2))^*=0$. By Riemann-Roch, $\chi(I_{l_1},I_{l_2})=-1$. Thus the lemma follows.
\end{proof}

Let $N\geq 1$ be an integer and $W$ be a complex vector space. Let $Q$ be the Quot-scheme of the quotient $W\otimes \cO_{V_5}(-N)\to E$ of $V_5$ with rank $2$ and Chern classes $c_1(E)=0$, $c_2(E)=2$ and $c_3(E)=0$ and $L$ the natural polarization \cite{Si}. The group $G=PGL(W)$ acts on $Q$ and a suitable power of $L$ is $G$-linearized. Let $Q_c^{ss}$ be the $PGL(W)$-semistable points corresponding to quotients without torsion and $Q_c$ the closure of $Q_c^{ss}$ in $Q$. When the integer $N$ and
the vector space $W$ are suitably chosen, the following properties are satisfied. The map $W\otimes \cO_{V_5}\to E(N)$ induce an isomorphism $W\simeq H^0(E(N))$ and $h^i(E(k))=0$ for $k\geq N$ and $i\geq 1$ and for all $E$ in $Q_c$. The point $[E]\in Q_c$ is semistable if and only if the sheaf $E$ is semistable if and only if $E\in Q_c^{ss}$. The stabilizer of $[E]$ in $GL (W)$ is identified with the group of
automorphisms of the sheaf $E$ and moduli space is then the GIT quotient $Q_c^{ss}//G$. There exist an open $G$-invariant subscheme $Q^{inst}\subset Q_c^{ss}$ parametrizing sheaves $E$ with $H^1(E(-1))=0$, and $M^{inst}$ is the GIT quotient $Q^{inst}//G$.
\begin{lem}\label{smooth3}
With the above hypothesis, the scheme $Q^{inst}$ is smooth.
\end{lem}
\begin{proof}
    The tangent space of $Q^{ss}_c$ at a point $[E]$ is isomorphic to $\Hom(F,E)$ where $F$ is the kernel of the map $W\otimes \cO_{V_5}(-N)\to E$. The scheme $Q_c^{ss}$ is smooth at the point if $\Ext^1(F,E)=0$. Consider the exact sequence:
    \begin{align*}
        \Ext^1(W\otimes \cO_{V_5}(-N),E)\to \Ext^1(F,E)\to \Ext^2(E,E)
    \end{align*}
    We then obtain an inclusion $\Ext^1(F,E)\to \Ext^2(E,E)$ since $h^1(E(N))=0$. To show the smoothness of $Q^{inst}$, it suffices then to prove $\Ext^2(E,E)=0$ for $[E]\in Q^{inst}$. If $E$ is stable and locally free, then $\Ext^2(E,E)=\Ext^2_{\mathcal{D}^b(Q_3)}(\Psi(E),\Psi(E))=0$ since the path algebra of a quiver is hereditary. If $E$ is stable but not locally free, we apply Lemma \ref{smooth1}. If $E$ is strictly semistable, then $E$ is the extension of $I_1$ and $I_2$, the vanishing follows from Lemma \ref{smooth 2}.
\end{proof}
\begin{thm}\label{smooth}
     The moduli space $M^{inst}$ of semistable sheaves of rank $2$ with Chern classes $c_1(E)=0,c_2(E)=2,c_3(E)=0$ satisfying $H^1(E(-1))=0$ on $V_5$ is smooth of dimension $5$.
\end{thm}
\begin{proof}
    See \cite[Theorem 4.6]{D}. Note $Q^{inst}$ is $G$-invariant and open in $Q^{ss}_c$. Let $x\in Q^{inst}$ and $E$ be the corresponding sheaf. Let $Q^{s-inst}\subset Q^{inst}$ be the set of stable instanton sheaves and $M^{s-inst}$ be the moduli space of stable instanton sheaves. The scheme $Q^{s-inst}$ is a principal $G$-space over $M^{s-inst}$ and hence $M^{s-inst}$ is smooth by the above lemma. It remains to study the case when $E=I_{l_1}\oplus I_{l_2}$ where $l_1,l_2$ are two lines in $V_5$. The orbit $O(x)$ of $x$ under $G$ is closed. The stabilizer $G_x$ of $x$ is a reductive group and there exist an affine subscheme  $W\subset Q^{ss}_c$ containing $x$ which is locally closed and stable under $G_x$ so that the morphism $W//G_x\to Q^{ss}_c//G$ is \'etale(\cite{L}). Let $(W, x)$ be the germ from $W$ to $x$ and let $F$ be the restriction to $X\times (W, x)$ of the tautological quotient on $X \times Q$. Then $((W,x),F)$ is a space of versal deformation for the sheaf $E$.(\cite{Og} Proposition 1.2.3). The germ $W$ is then smooth at $x$ by Lemma \ref{smooth 2}. Since the morphism $W//G_x\to Q^{ss}_c//G$ is \'etale, it suffices to show the quotient $W//G_x$ is smooth at $[x]$. Now there exist a $G_x$ linear morphism $(W,x)\to (T_xW,0)$ \'etale at $x$ so that the induced morphism $W//G_x\to T_xW//G_x$ is \'etale at $[x]$(\cite{L}). It is therefore sufficient to prove that the quotient $T_xW//G_x$ is smooth at 0.\
    
    Suppose $l_1$ and $l_2$ are distinct. The tangent space $T_xW=\Ext^1(E,E)$ is of dimension $6$ and $G_x=G_m\times G_m$ acts on the space by formula (\cite{Og} Lemma 1.4.16)
    \begin{align*}
        (t,t')\cdot(\sum_{i,j}e_{i,j})=e_{1,1}+t/t'e_{1,2}+t'/te_{2,1}+e_{2,2}.
    \end{align*}
    It is easy to verify the quotient $T_xW//G_x$ is the affine space $\mathbb{A}^5$ and in particular smooth at $0$.\
    
    Suppose $l_1$ and $l_2$ are the same and use $l$ to denote the line. The tangent space $\Ext^1(E,E)$ is then of dimension $8$ and $G_x=\mathrm{PGL}(2)$. Let $T=\Ext^1(I_l,I_l)$ and let $U$ be a vector space of dimension $2$. The group $G_x$ acts on $T_xW=T\otimes \mathrm{End}(U)$ by conjugation on $\mathrm{End}(U)$(\cite{Og} Lemma 1.4.6) The quotient $T_x//G_x$ is then isomorphic to $\mathbb{A}^5$(\cite{La}, III, case 2) and in particular smooth at $0$.
\end{proof}
Since $\mathcal{U}$ and $\mathcal{Q}\smvee$ both lie in $\mathcal{B}_{V_5}$, we see $[E]\in M$ lies in the image of  $\phi$ only if $H^1(E(-1))=0$. Thus we have a morphism $\phi:M_{rep}\to M^{inst}$.
\begin{thm}\label{mainthm}
    $\phi:M_{rep}\to M^{inst}$ is an isomorphism. As a result, the moduli space of semistable rank $2$ sheaves with Chern classes $c_1=0$, $c_2=2$ and $c_3=0$ on $V_5$ has a connected component isomorphic to $\pP^5$.
\end{thm}
\begin{proof}
    Since both $M_{rep}$ and $M$ are projective, $\phi: M_{rep}\to M$ is proper. Since $M^{inst}$ is an open subscheme of $M$, $\phi:M_{rep}\to M^{inst}$ is proper. By Lemma \ref{injective}, $\phi$ is injective. By Theorem \ref{rep2} and Proposition \ref{ss}, $\phi $ is surjective. As a result, $M^{inst}$ is connected.
    
    Let $[R]\in M_{rep}$ be a stable representation, then the tangent space at $[R]$ is given by $\Ext_{Q_3}^1(R,R)$. The tangent space at $\phi(R)$ (which is also stable) is given by $\Ext^1(\Psi^{-1}(R),\Psi^{-1}(R))$.
    Since $\Psi:\mathcal{B}_{V_5}\to\mathcal{D}^b(Q_3)$ is an equivalence, $\phi$ induces isomorphism between tangent spaces, so $\phi$ is \'etale when restricted to the open stable locus. Since $\phi$ is also injective and we are working over $\C$, $\phi$ is an open immersion over the stable locus (See for example Stack Project $40.14$). Thus $\phi$ is birational. 
    Now $\phi$ is a bijective birational proper morphism, it has to be an isomorphism.\
    
    At last, since $\phi$ is proper and $M^{inst}\cong \pP^5$ is its image, $M^{inst}\subset M$ is closed. But $M^{inst}$ is also open, hence $M^{inst}\cong \pP^5$ is a connected component of $M$.
\end{proof}
\begin{rem}
\noindent
\begin{enumerate}
    \item  Theorem \ref{mainthm} shows that a natural compactification of the moduli space of minimal instanton bundles on $V_5$ is $\pP^5$. 
    \item If Conjecture \ref{conj} is true, the complement of $M^{inst}$ in $M$ is actually empty. Then the moduli space of semistable rank $2$ sheaves with Chern classes $c_1=0$, $c_2=2$ and $c_3=0$ on $V_5$ is isomorphic to $\pP^5$.
\end{enumerate}
\end{rem}

\end{document}